\documentclass[11pt]{amsart}
\usepackage{amscd,amssymb}
\usepackage{amsthm,amsmath}

\usepackage{float}
\usepackage{pgf,tikz}
\usepackage{mathrsfs}
\usetikzlibrary{arrows}

\usepackage{fancyhdr}
\usepackage{supertabular}
\usepackage{setspace}
\usepackage{color}
\usepackage{longtable}

\usepackage[matrix,arrow,curve]{xy}
\usepackage{graphicx}
\usepackage{enumerate}
\usepackage{amsfonts}
\usepackage{cite}
\usepackage{pifont}
\usepackage{multirow}
\usepackage{hhline}

\newtheorem{theorem}[equation]{Theorem}
\newtheorem*{maintheorem*}{Main Theorem}
\newtheorem{lemma}[equation]{Lemma}

\newtheorem{corollary}[equation]{Corollary}
\newtheorem{conjecture}[equation]{Conjecture}

\newtheorem*{question*}{Question}

\newtheorem*{problem*}{Problem}
\newtheorem*{theoremA*}{Theorem A}
\newtheorem*{theoremB*}{Theorem B}
\newtheorem*{theoremC*}{Theorem C}

\theoremstyle{definition}
\newtheorem{example}[equation]{Example}
\newtheorem{definition}[equation]{Definition}

\theoremstyle{remark}
\newtheorem{remark}[equation]{Remark}

\makeatletter\@addtoreset{equation}{section} \makeatother

\title{Super-rigid affine Fano varieties}

\author{Ivan Cheltsov, Adrien Dubouloz, Jihun Park}

\address{ \emph{Ivan Cheltsov}\newline \textnormal{School of Mathematics, The University of Edinburgh
\newline \medskip James Clerk Maxwell Building, The King's Buildings, Mayfield Road, Edinburgh EH9 3JZ, UK.\newline
\texttt{I.Cheltsov@ed.ac.uk}}}

\address{\emph{Adrien Dubouloz}\newline \textnormal{IMB UMR5584, CNRS, Univ. Bourgogne-Franche Comt\'e, F-21000 Dijon, France.
\newline \texttt{adrien.dubouloz@u-bourgogne.fr}}}
\begin{document}

\address{ \emph{Jihun Park}\newline \textnormal{Center for Geometry and Physics, Institute for Basic Science (IBS)
\newline \medskip 77 Cheongam-ro, Nam-gu, Pohang, Gyeongbuk, 37673, Korea.
\newline Department of Mathematics, POSTECH
\newline 77 Cheongam-ro, Nam-gu, Pohang, Gyeongbuk,  37673, Korea. \newline
\texttt{wlog@postech.ac.kr}}}

\begin{abstract}
We study a wide class of affine varieties, which we call affine Fano varieties.
By analogy with birationally super-rigid Fano varieties, we define super-rigidity for affine Fano varieties,
and provide many examples and non-examples of super-rigid affine Fano varieties.
\end{abstract}

\maketitle

\section{Introduction}
\label{section:intro}
Throughout the present article, all varieties and morphisms  are assumed to be defined over a  field $\Bbbk$ of characteristic zero.

 The automorphism group of an affine surface $S$ admitting a smooth projective completion $X$ whose boundary consists of  a single smooth curve $C$ has been intensively studied during the last decades \cite{GizatulinDanilov1975,GizatulinDanilov1976,DuLa15, BD11,BD15}, inspired by the pioneering work of Gizatullin \cite{Giza70}. The upshot is that depending on whether $C$ is rational or not, the group $\mathrm{Aut}(S)$ is either an infinite dimensional group which sometimes cannot even be generated by any countable family of its algebraic subgroups, or an affine algebraic group isomorphic to the group $\mathrm{Aut}(X,C)$ of automorphisms of the pair $(X,C)$. Not much is known in higher dimension, except for the general fact if a smooth affine variety $V$ admits a smooth projective completion $X$ whose boundary $S_X$ is a smooth non-ruled hypersurface of $X$, then the restriction map $\mathrm{Aut}(X,S_X)\rightarrow \mathrm{Aut}(X\setminus S_X)$ is again an isomorphism. For affine cubic hypersurfaces, there is a folklore:

\begin{conjecture}
\label{conjecture:cubics}
Let $X$ be a smooth cubic hypersurface in the complex projective space~$\mathbb{P}^4$,
and let $S_X$ be its hyperplane section. The complement $X\setminus S_X$ is an affine cubic hypersurface in~$\mathbb{A}^4$.
Suppose that the cubic surface $S_X$ is smooth.
Then
$$
\mathrm{Aut}\left(X\setminus S_X\right)=\mathrm{Aut}\left(X,S_X\right).
$$
In particular, the group $\mathrm{Aut}(X\setminus S_X)$ is finite.
\end{conjecture}

In this article, we study a similar problem for a wide class of affine varieties, which we call \emph{affine Fano varieties}.
Namely, let $X$ be a projective normal variety with $\mathbb{Q}$-factorial singularities such that its Picard group is of rank $1$,
and let $S_X$ be a prime divisor on $X$ such that the  pair $(X,S_X)$ has purely log terminal singularities.
In particular, the subvariety $S_X$ is normal. Since the divisor $S_X$ is ample, the set $X\setminus S_X$ is an affine variety.

\begin{definition}
\label{definition:affine-Fano}
If $-(K_X+S_X)$ is an ample divisor,
then the affine variety $X\setminus S_X$ is said to be \emph{an affine Fano variety} with completion $X$ and boundary $S_X$.
\end{definition}

By analogy with Mori fiber spaces in birational geometry, we introduce the following relative version of Definition~\ref{definition:affine-Fano}.

\begin{definition}
\label{definition:affine-Fano-relative}
Let $\mathcal{X}$ and $\mathcal{B}$ be quasi-projective normal varieties
such that there exists a dominant non-birational projective morphism $\rho\colon\mathcal{X}\to\mathcal{B}$ with connected fibers,
and let~$\mathcal{S}_\mathcal{X}$ be a prime divisor on $\mathcal{X}$.
Denote by $\mathcal{X}_\eta$ the scheme theoretic generic fiber of $\rho$,
so that $\mathcal{X}_\eta$ is defined over the function field of $\mathcal{B}$.
Put $\mathcal{S}_{\mathcal{X}_\eta}=\mathcal{S}_{\mathcal{X}}\vert_{\mathcal{X}_\eta}$.  We say that $\mathcal{X}\setminus \mathcal{S}_\mathcal{X}$ is a (relative) \emph{affine Fano variety} over $\mathcal{B}$ provided that
the following properties are satisfied
\begin{itemize}
\item the generic fiber  $\mathcal{X}_\eta$ is a projective variety of Picard rank $1$ with $\mathbb{Q}$-factorial singularities;
\item the  pair $(\mathcal{X}_\eta, \mathcal{S}_{\mathcal{X}_\eta})$ has purely log terminal singularities;
\item the divisors $-(K_{\mathcal{X}_\eta}+\mathcal{S}_{\mathcal{X}_\eta})$ and $\mathcal{S}_{\mathcal{X}_\eta}$ are ample.
\end{itemize}
\end{definition}

\label{remark:cylinders} This definition comprises two important notions: ``usual'' affine Fano varieties as in Definition~\ref{definition:affine-Fano}
when the base $\mathcal{B}$ is a point, and $\mathbb{A}^1$-\emph{cylinders} when $\mathrm{dim}(\mathcal{B})>0$.  Recall that an $\mathbb{A}^1$-cylinder is a variety isomorphic to $\mathcal{B}\times \mathbb{A}^1$ for some smooth quasi-projective variety~$\mathcal{B}$. In this case, the projection $\mathrm{pr}_B\colon \mathcal{X}=\mathcal{B}\times \mathbb{P}^1\to\mathcal{B}$ is a relative affine Fano variety with respect to the divisor $\mathcal{S}_{\mathcal{X}}=\mathcal{B}\times\{\infty\}$, where $\infty=\mathbb{P}^1\setminus\mathbb{A}^1$.
Projective varieties that contain Zariski open $\mathbb{A}^1$-cylinders appear naturally in many problems and questions (see \cite{Ru81,KMc99,KPZ13,DuKi15,CPW16,CPW13}). Some of them are still open. For instance, it is not known whether every
smooth projective rational variety always contains a Zariski open $\mathbb{A}^1$-cylinder or not.

In dimension one, the only affine Fano variety is the affine line $\mathbb{A}^1$,
and its automorphisms are induced by the automorphisms of its completion $\mathbb{P}^1$.
This is no longer true in higher dimensions: the complement $\mathbb{A}^2$ of a line $L$ in $\mathbb{P}^2$ is an affine Fano variety that contains an open $\mathbb{A}^1$-cylinder, for which $\mathrm{Aut}(\mathbb{A}^2) \ne \mathrm{Aut}(\mathbb{P}^2,L)$.
Keeping in mind Conjecture~\ref{conjecture:cubics} and the fact that smooth cubic threefolds in $\mathbb{P}^4$ do not contain open $\mathbb{A}^1$-cylinders (see \cite{DuKi15}) we are primarily interested in affine Fano varieties with the following properties:

\begin{enumerate}
\item Their automorphisms are induced by automorphisms of their completions.
\item They do not contain proper relative affine Fano varieties over varieties of positive dimension.
In particular, they do not contain open $\mathbb{A}^1$-cylinders.
\end{enumerate}

These resemble basic properties of the so-called \emph{birationally super-rigid} Fano varieties.
Namely, recall from \cite[Definition~1.3]{Co00} that a Fano variety $V$ with terminal $\mathbb{Q}$-factorial singularities
and  Picard rank $1$ is said to be birationally super-rigid if the following two conditions hold.
\begin{enumerate}
\item For every Fano variety $V^\prime$ with terminal $\mathbb{Q}$-factorial singularities and Picard rank~$1$,
if there exists a birational map $\phi\colon V\dasharrow V^\prime$, then $\phi$ is an isomorphism.
In particular, one has
$$
\mathrm{Bir}\left(V\right)=\mathrm{Aut}\left(V\right).
$$
\item The variety $V$ is not birational to a fibration into Fano varieties over varieties of positive dimensions.
In particular, the variety $V$ is not birationally ruled.
\end{enumerate}
The first example of birationally super-rigid Fano variety was given by Iskovskikh and Manin in \cite{IM71}, where they implicitly showed that every smooth quartic threefold is birationally super-rigid.
Since then birational super-rigidity has been proved for many higher-dimensional Fano varieties (see \cite{CPR,Ch05,Pukhlikov1998,Fernex,Pukhlikov,CP17}).

By analogy with birational super-rigidity Fano varieties, we introduce

\begin{definition}
\label{definition:rigid}
An affine Fano variety $X\setminus S_X$ with completion $X$ and boundary $S_X$
is said to be \emph{super-rigid} if the following  two conditions hold.
\begin{enumerate}
\item For every affine Fano variety $X^\prime\setminus S^\prime_{X^\prime}$ with completion $X^\prime$ and boundary $S^\prime_{X^\prime}$,
if there exists an isomorphism $\phi\colon X\setminus S_X\cong X^\prime\setminus S^\prime_{X^\prime}$,
then $\phi$ is induced by an isomorphism $X\cong X^\prime$ that maps $S_X$ onto $S^\prime_{X^\prime}$.
In particular, one has
$$
\mathrm{Aut}\left(X\setminus S_X\right)=\mathrm{Aut}\left(X,S_X\right).
$$
\item The affine Fano variety $X\setminus S_X$ does not contain relative affine Fano varieties of the same dimension over varieties of positive dimensions.
In particular, $X\setminus S_X$ does not contain open $\mathbb{A}^1$-cylinders.
\end{enumerate}
\end{definition}

The arguably simplest example of a super-rigid affine Fano surface is given by

\begin{example}[{\cite[Example~3]{DuKi15}}]
Let $C$ be an irreducible conic in $\mathbb{P}^2$. Then $\mathbb{P}^2\setminus C$ is an affine Fano variety.
If $C$ has a rational point, then $\mathbb{P}^2\setminus C$ contains an open $\mathbb{A}^1$-cylinder, so that $\mathbb{P}^2\setminus C$ is  not super-rigid. On the other hand, if the conic $C$ does not have a rational points,
then $\mathbb{P}^2\setminus C$ is super-rigid.
\end{example}

As in birational super-rigidity, in general  it is arduous  to determine whether a given affine Fano variety is super-rigid or not. For a terminal $\mathbb{Q}$-factorial Fano variety of Picard rank $1$, the following theorem, known as a classical N\"other-Fano inequality, serves as a criterion for birational super-rigidity.
\begin{theorem} Let $V$ be a terminal $\mathbb{Q}$-factorial Fano variety of Picard rank $1$.
If for any mobile linear system $\mathcal{M}_V$ on $V$ and a positive rational number $\lambda$ such that
$$
K_V+\lambda\mathcal{M}_V\sim_{\mathbb{Q}} 0,
$$
the pair $(V,\lambda \mathcal{M}_V)$ is canonical, then $V$ is super-rigid.
\end{theorem}

In Section~\ref{section:alpha}, we will establish a similar statement for  super-rigid affine Fano varieties as below.
\begin{theoremA*}
 Let $X\setminus S_X$ be an affine Fano variety with completion $X$ and boundary $S_X$.
If for any mobile linear system $\mathcal{M}_X$ on $X$ and a positive rational number $\lambda$ such that
$$
K_X+S_X+\lambda\mathcal{M}_X\sim_{\mathbb{Q}} 0,
$$
the pair $(X,S_X+\lambda \mathcal{M}_X)$ is log canonical, then $X\setminus S_X$ is super-rigid.
\end{theoremA*}

From the viewpoint of the classical N\"other-Fano inequality, a terminal $\mathbb{Q}$-factorial Fano variety~$V$ of Picard rank $1$ is birationally super-rigid provided that
\begin{quote}
 (*) the pair $(V, D)$ is canonical for every effective $\mathbb{Q}$-divisor $D$ that is $\mathbb{Q}$-linearly equivalent to $-K_V$.
 \end{quote}
 Note that this condition is not a necessary condition for birational super-rigidity.

In order to introduce analogy with this sufficient condition, we use
 a generalized version of Tian's $\alpha$-invariant, which was introduced in \cite{Tian} using a different language.
Let $X\setminus S_X$  be an affine Fano variety with completion $X$ and boundary $S_X$.
Then there exists a uniquely determined effective $\mathbb{Q}$-divisor $\mathrm{Diff}_{S_X}(0)$ on $S_X$,
usually called a \emph{different}, which was introduced by Shokurov in \cite{Sho93}.
Namely, let $R_{1},\ldots,R_{s}$ be all irreducible components of the locus $\mathrm{Sing}(X)$ of dimension $\mathrm{dim}(X)-2$ that are contained in $S_X$. Then
$$
\mathrm{Diff}_{S_X}\left(0\right)=\sum_{i=1}^{s}\frac{m_{i}-1}{m_{i}}R_{i},
$$
where $m_{i}$ is the~smallest positive integer such that $m_{i}S_X$ is~Cartier at a~general point of the divisor $R_{i}$.
By the adjunction formula, we have
$$
K_{S_X}+\mathrm{Diff}_{S_X}(0)\sim_{\mathbb{Q}} (K_X+S_X)\vert_{S_X}.
$$
Moreover, by Adjunction (see \cite[Theorem~7.5]{Ko97}), the  pair $(S_X,\mathrm{Diff}_{S_X}(0))$ has Kawamata log terminal singularities.
Thus, since  $X\setminus S_X$ is an affine Fano variety, the pair $(S_X,\mathrm{Diff}_{S_X}(0))$ is a log Fano variety in a usual sense.

The $\alpha$-invariant of the pair $(S_X,\mathrm{Diff}_{S_X}(0))$ is defined as
\[\alpha(S_X,\mathrm{Diff}_{S_X}(0))=
\mathrm{sup}\left\{\ \lambda \ \left| \
 \aligned
& \mbox{the  pair } (S_X, \mathrm{Diff}_{S_X}(0)+\lambda D) \mbox{ is log canonical for every}\\
& \mbox{effective $\mathbb{Q}$-divisor $D$ such that } D\sim_{\mathbb{Q}} -\left(K_{S_X}+\mathrm{Diff}_{S_X}(0)\right)\\
\endaligned\right.\right\}.
\]

Then the condition $\alpha(S_X,\mathrm{Diff}_{S_X}(0))\geqslant 1$ will be an analogy of the condition (*).

\begin{theoremB*}
Let $X\setminus S_X$ be an affine Fano variety with completion $X$ and boundary~$S_X$.
If $\alpha(S_X,\mathrm{Diff}_{S_X}(0))\geqslant 1$, then the affine Fano variety $X\setminus S_X$ is super-rigid.
\end{theoremB*}

This sufficient condition allows to give large supply of super-rigid affine Fano varieties in every dimension bigger than or equal to $3$.
As in birational super-rigidity, it can scarcely be expected that the condition in Theorem~B is  sufficient and necessary to be super-rigid. However our results (\textbf{Theorems~\ref{theorem:quotient-n-2}, \ref{theorem:quotient-n-3} and~\ref{theorem:quotient-n-4}}) show that it  plays a role of a criterion for affine Fano varieties of a certain type to be super-rigid.

With Theorem~B and some result on non-rationality of Fano varieties we can provide an instructive example that fosters our understanding of Conjecture~\ref{conjecture:cubics}. The example below stands in the same stream as Conjecture~\ref{conjecture:cubics}, with slightly less difficulty.

\begin{example}
\label{example:Grinenko}
Let $X$ be a smooth hypersurface in $\mathbb{P}(1,1,1,2,3)$ of degree $6$ and $S_X$ be a hyperplane section of $X$ by an equation of degree $1$ with at most Du Val singularities. The hyperplane section $S_X$ is a del Pezzo surface of anticanonical degree $1$. The pair $(X,S_X)$ has purely log terminal singularities, and hence $X\setminus S_X$ is an affine Fano variety.
It follows from \cite{Ch07a, ChKo14} that $\alpha(S_X,\mathrm{Diff}_{S_X}(0))=1$ if the following two conditions are satisfied:
\begin{enumerate}
\item the surface $S_X$ has only singular points of types $\mathrm{A}_1$, $\mathrm{A}_2$ and $\mathrm{A}_3$;
\item the linear system $|-K_{S_X}|$ does not contain cuspidal curves.
\end{enumerate}
In case, Theorem~B immediately implies that $X\setminus S_X$ is super-rigid. We do not know in general if $X\setminus S_X$ is super-rigid or not. We may draw only a conclusion that $\mathrm{Aut}(X\setminus S_X)=\mathrm{Aut}(X,S_X)$ regardless of how singular the surface $S_X$ is. This follows from the result of Grinenko (\cite{Grinenko1,Grinenko2}) that every birational automorphism of $X$ is biregular, i.e., $\mathrm{Bir}(X)=\mathrm{Aut}(X)$.
\end{example}

It is natural that there should be a lot of mysteries in the border area between super-rigidity and non-super-rigidity. For those affine Fano varieties $X\setminus S_X$ near the border even the problem to determine  whether the groups $\mathrm{Aut}(X\setminus S_X)$ and $\mathrm{Aut}(X,S_X)$ coincide or not is subtle.
We believe that smooth cubic threefolds lie in this border area from the super-rigidity side.
This is one of the reasons why Conjecture~\ref{conjecture:cubics} is far away from its proof.

Meanwhile, as in birational super-rigidity, the super-rigidity of an affine Fano variety $X\setminus S_X$ may be easily destroyed if we allow some singularities in the boundary $S_X$. This has been partially observed and investigated in \cite{KPZ11, DuKi15} from singular cubic surfaces. In Section~~\ref{section:complements} we provide a complete and uniform picture for the complements of del Pezzo surfaces in the context of affine Fano varieties. We obtain in particular  the following characterization of non super-rigid complement of del Pezzo surfaces:

\begin{theoremC*}\label{theoremC} Let $S$ be a del Pezzo hypersurface of anticanonical degree at most $3$ with Du Val singularities in a (weighted) projective space $\mathbb{P}$. If the surface $S$ contains an open $(-K_S)$-cylinder, then $\mathbb{P}\setminus S$ is an affine Fano variety that contains an open $\mathbb{A}^1$-cylinder. In particular, it is not super-rigid.
\end{theoremC*}
Here we recall from \cite[Definition~1.3]{CPW16} that an open $\mathbb{A}^1$-cylinder $U$ in a del Pezzo surface $S$ is said to be $(-K_S)$-\emph{polar} if $S\setminus U$ is the support of an effective divisor $\mathbb{Q}$-linearly equivalent to the anti-canonical divisor $-K_S$ of $S$.

\section{Log N\"other-Fano Inequalities and $\alpha$-invariant of Tian}
\label{section:alpha}

In this section, we will establish Theorems~A and~B. For this goal we generalise the classical N\"other-Fano inequality for affine Fano varieties $X\setminus S_X$ in the context of purely log terminal pairs $(X, S_X)$.

First, we suppose that there is an affine Fano variety $Y\setminus S_Y$ and a birational map $\phi\colon X\dasharrow Y$ such that
\begin{enumerate}
\item the map $\phi$ is not an isomorphism in codimension $1$;
\item the map $\phi$ induces an isomorphism of $X\setminus S_X$ onto $Y\setminus S_Y$.
\end{enumerate}
Let $\mathcal{M}_Y$ be a very ample  complete  system on $Y$, and let $\mathcal{M}_X$ be the proper transform of $\mathcal{M}_Y$ by~$\phi^{-1}$.
We consider the following resolution of indeterminacy of $\phi$
$$
\xymatrix{
&&W\ar@{->}[ld]_{f}\ar@{->}[rd]^{g}&&\\%
&X\ar@{-->}[rr]_{\phi}&&Y,&}
$$ %
where $f$ and $g$ are birational morphisms and $W$ is a smooth projective variety.
Let $\widetilde{S}_X$ and $\widetilde{S}_Y$ be proper transforms of $S_X$ and $S_Y$ by $f$ and $g$, respectively.
Due to the conditions above, the divisor $\widetilde{S}_X$ is $g$-exceptional,
the divisor $\widetilde{S}_Y$ is $f$-exceptional, and $f(\widetilde{S}_Y)$ is contained in $S_X$.

By analogy with the notion of $\epsilon$-Kawamata log terminal singularities, we say that a pair $(Y,S_Y)$ with irreducible and reduced boundary $S_Y$ has $\epsilon$-purely log terminal singularities, if for every resolution of singularities $h\colon Z\rightarrow Y$, where $Z$ is a smooth projective variety, the discrepancy of every exceptional divisor of $h$ is bigger than $-1+\epsilon$.

\begin{lemma}[Log N\"other-Fano Inequality I]
\label{Lemma:NFI-1}
Suppose that $(Y,S_Y)$ has $\epsilon$-purely log terminal for some positive rational number $\epsilon$.
Let $\mu$ and $\lambda$  be non-negative  rational numbers such that $ 1-\epsilon \leqslant \mu \leqslant 1$ and
$$
K_X+\mu S_X+\lambda\mathcal{M}_X\sim_{\mathbb{Q}} 0.
$$
Then the log discrepancy of  $\widetilde{S}_Y$ with respect to  $(X,\mu S_X+\lambda\mathcal{M}_X)$  is less than $1-\mu$.
\end{lemma}

\begin{proof}
We have
$$
K_W+\mu\widetilde{S}_X+\lambda\mathcal{M}_W\sim_{\mathbb{Q}}f^*(K_X+\mu S_X+\lambda\mathcal{M}_X)+a\widetilde{S}_Y+\sum a_iE_i,
$$
and
$$
K_W+\mu\widetilde{S}_Y+\lambda\mathcal{M}_W\sim_{\mathbb{Q}}g^*(K_Y+\mu S_Y+\lambda\mathcal{M}_Y)+b\widetilde{S}_X+\sum b_iE_i,
$$
where each $E_i$ is simultaneously $f$-exceptional and $g$-exceptional. The hypothesis that $(Y,S_Y)$ is $\epsilon$-purely log terminal immediately implies that $\mu+b$ is positive.

Suppose $a+\mu\geqslant 0$. Since  $\mu+b$ is positive, applying Negativity Lemma (\cite[1.1]{Sho93}) to
$$
(\mu+a)\widetilde{S}_Y-g^*(K_Y+\mu S_Y+\lambda\mathcal{M}_Y)\sim_{\mathbb{Q}}(\mu+b)\widetilde{S}_X+\sum (b_i-a_i)E_i
$$
we see that $K_Y+\mu S_Y+\lambda\mathcal{M}_Y$ is positive.

Let $\lambda^\prime$ be the number such that $K_Y+\mu S_Y+\lambda\prime\mathcal{M}_Y\sim_{\mathbb{Q}} 0$. Note that $\lambda^\prime<\lambda$.
We have
$$
K_W+\mu\widetilde{S}_X+\lambda^\prime\mathcal{M}_W=f^*(K_X+\mu S_X+\lambda^\prime\mathcal{M}_X)+a^\prime\widetilde{S}_Y+\sum a_i^\prime E_i,
$$
and
$$
K_W+\mu\widetilde{S}_Y+\lambda^\prime\mathcal{M}_W=g^*(K_Y+\mu S_Y+\lambda^\prime\mathcal{M}_Y)+b\widetilde{S}_X+\sum b_iE_i.
$$
where $a^\prime\geqslant a$ and $a^\prime_i\geqslant a_i$.
We then get
$$
(\mu+b)\widetilde{S}_X-f^*(K_X+\mu S_X+\lambda^\prime\mathcal{M}_X)\sim_{\mathbb{Q}} (\mu+a^\prime)\widetilde{S}_Y+\sum (a_i^\prime-b_i)E_i.
$$
Since $\mu+b>0$ and $K_X+\mu S_X+\lambda^\prime\mathcal{M}_X$ is negative,  Negativity Lemma implies $\mu+a^\prime\leq 0$.
This is absurd. Therefore, $a+1<1-\mu$.
\end{proof}

\begin{corollary}
\label{corollary:NFI-1}
Let $\lambda$  be a non-negative  rational number such that
$$
K_X+S_X+\lambda\mathcal{M}_X\sim_{\mathbb{Q}} 0.
$$
Then the log discrepancy of $\widetilde{S}_Y$ with respect to the  pair $(X,S_X+\lambda\mathcal{M}_X)$  is less than~$0$.
\end{corollary}

Now let $Y$ be a quasi-projective variety, let $S_Y$ be a prime divisor on the variety $Y$, and let  $\rho\colon Y\to B$
be a dominant  non-birational projective morphism with connected fibers such that
the complement $Y_\eta\setminus S_{Y_\eta}$ is an affine Fano variety over the function field of the variety $B$.
Here $Y_\eta$ is a generic fiber of $\rho$, and $S_{Y_\eta}=S_Y\vert_{Y_\eta}$.
We projectivize $Y$ and $B$ into projective varieties $\overline{Y}$ and $\overline{B}$ respectively.
We may assume that $\rho\colon Y\to B$ extends to a projective morphism $\bar{\rho}\colon\overline{Y}\to\overline{B}$.

Suppose that $\dim(B)>0$, and that there exists a birational map $\psi\colon\overline{Y}\dasharrow X$
that induces an embedding  $Y\setminus S_Y$ into $X\setminus S_X$.
As above, we consider the following resolution of indeterminacy of $\psi$
$$
\xymatrix{
&&V\ar@{->}[ld]_{p}\ar@{->}[rd]^{q}&&\\%
&X&&\overline{Y}\ar@{-->}[ll]_{\psi}\ar@{->}[d]^{\bar{\rho}}&\\
&&&\overline{B}&}
$$ %
where $p$ and $q$ are birational morphisms and $V$ is a smooth projective variety.
Let $\overline{S}_Y$ be the closure of $S_Y$ in $\overline{Y}$.
Let $\widetilde{S}_X$ and $\widetilde{S}_Y$ be the proper transforms of $S_X$ and $\overline{S}_Y$ by $p$ and $q$, respectively.
Then $\widetilde{S}_Y$ is $p$-exceptional, because $X$ has $\mathbb{Q}$-factorial singularities, and its Picard group is of rank $1$.
Moreover, since $\psi$ induces an embedding  $Y\setminus S_Y$ into $X\setminus S_X$,
either $q(\widetilde{S}_X)$ is contained in $\overline{S}_Y$, or $q(\widetilde{S}_X)$ is contained in $\overline{Y}\setminus Y$.
Thus, either the divisor $\widetilde{S}_X$ is $q$-exceptional, or $\widetilde{S}_X$ is mapped by $\bar{\rho}\circ q$
to a proper subvariety of $\overline{B}$.
Since $S_X$ is ample, the latter also implies that $p(\widetilde{S}_Y)$ is contained in $S_X$.

Let $H$ be a very ample Cartier divisor on $\overline{B}$.
Let $\mathcal{H}_{\overline{Y}}$ be the complete linear system $|\bar{\rho}^*(H)|$,
and let $\mathcal{H}_X$ be the proper transform of $\mathcal{H}_{\overline{Y}}$ on the variety $X$ by $\psi$.

\begin{lemma}[Log N\"other-Fano Inequality II]
\label{lemma:NFI-2}
Let $\mu$ and $\lambda$ be non-negative  rational numbers such that $\mu\leqslant 1$ and
$$
K_X+\mu S_X+\lambda\mathcal{H}_X\sim_{\mathbb{Q}} 0.
$$
Then the log discrepancy of $\widetilde{S}_Y$ with respect to $(X, \mu S_X+\lambda  \mathcal{H}_X)$  is less than $1-\mu$.
\end{lemma}

\begin{proof}
As in the proof of Lemma~\ref{Lemma:NFI-1}, we have
$$
K_V+\mu\widetilde{S}_X+\lambda\mathcal{H}_V\sim_{\mathbb{Q}} p^*(K_X+\mu S_X+\lambda\mathcal{H}_X)+a\widetilde{S}_Y+\sum a_iE_i,
$$
and
$$
K_V+\mu\widetilde{S}_Y+\lambda\mathcal{H}_V\sim_{\mathbb{Q}} q^*(K_{\overline{Y}}+ \mu \overline{S}_Y+\lambda\mathcal{H}_{\overline{Y}})+b\widetilde{S}_X+\sum b_jF_j,
$$
where each $E_i$ is $p$-exceptional, and each $F_j$ is $q$-exceptional.
Then
$$
(a+\mu)\widetilde{S}_Y\sim_{\mathbb{Q}} q^*(K_{\overline{Y}}+ \mu\overline{S}_Y+\lambda\mathcal{H}_{\overline{Y}})+(b+\mu)\widetilde{S}_X+\sum b_j F_j-\sum a_i E_i.
$$
Thus, we have
$$
(a+\mu)\overline{S}_Y\sim_{\mathbb{Q}} K_{\overline{Y}}+ \mu\overline{S}_Y+\lambda\mathcal{H}_{\overline{Y}}+(b+\mu)q(\widetilde{S}_X)-\sum a_i q(E_i).
$$
On the other hand, each $p$-exceptional divisor $E_i$ is either $q$-exceptional, or $E_i$ is mapped by $\bar{\rho}\circ q$ to a proper subvariety of $\overline{B}$.
This shows that
$$
-(a+\mu)S_{Y_\eta}\sim_{\mathbb{Q}} -(K_{Y_\eta}+\mu S_{Y_\eta})=-(K_{Y_\eta}+S_{Y_\eta})+\left(1-\mu\right)S_{Y_\eta}.
$$
 Since $-(K_{Y_\eta}+S_{Y_\eta})$ and $S_{Y_\eta}$  are ample, we obtain $a+\mu<0$. Therefore, the log discrepancy of $\widetilde{S}_Y$ with respect to $(X,\mu S_X+\lambda\mathcal{H}_X)$ is less than $1-\mu$.
\end{proof}

\begin{corollary}
\label{corollary:NFI-2}
Let $\lambda$ be a non-negative rational number such that
$$
K_X+S_X+\lambda\mathcal{H}_X\sim_{\mathbb{Q}} 0.
$$
Then the log discrepancy of $\widetilde{S}_Y$ with respect to the  pair $(X,S_X+\lambda\mathcal{H}_X)$  is less than~$0$.
\end{corollary}

\begin{theorem}[Theorem~A]
\label{theorem:NF-inequality} Let $X\setminus S_X$ be an affine Fano variety with completion $X$ and boundary $S_X$.
If for any mobile linear system $\mathcal{M}_X$ on $X$ and a positive rational number $\lambda$ such that
$$
K_X+S_X+\lambda\mathcal{M}_X\sim_{\mathbb{Q}} 0,
$$
the pair $(X,S_X+\lambda \mathcal{M}_X)$ is log canonical along $S_X$, then $X\setminus S_X$ is super-rigid.
\end{theorem}

\begin{proof}
This immediately follows from Corollaries~\ref{corollary:NFI-1} and \ref{corollary:NFI-2}.
\end{proof}

As a application of Theorem~A, we obtain the following positive solution
to Conjecture \ref{conjecture:cubics} for smooth cubic surfaces without rational points.

\begin{corollary}
\label{corollary:cubics}
Let $S$ be a smooth cubic surface in $\mathbb{P}^3_{\Bbbk}$ without $\Bbbk$-rational points.
Then $\mathbb{P}^3_{\Bbbk}\setminus S$ is a super-rigid affine Fano variety.
\end{corollary}

\begin{proof}
Indeed, otherwise Theorem~A implies that there exists a mobile linear system $\mathcal{M}$ on $\mathbb{P}^3$,
say of degree $m>2$ such that the singularities of the  pair
$$
\left(\mathbb{P}^3, S+\frac{1}{m}\mathcal{M}\right)
$$
are not log canonical along $S$. Write $D=\frac{1}{m}\mathcal{M}|_S$.
Then the pair $(S,D)$ is not log canonical by Inversion of adjunction,
and $D$ is an effective $\mathbb{Q}$-divisor on $S$ such that $D\sim_{\mathbb{Q}} -K_S$.

Let $S_{\overline{\Bbbk}}$ be the surface obtained from $S$ by base change to an algebraic closure $\overline{\Bbbk}$ of $\Bbbk$.
Similarly, let $D_{\overline{\Bbbk}}$ be the $\mathbb{Q}$-divisor  obtained from $D$ by the same base change.
Then $D_{\overline{\Bbbk}}$ is an effective $\mathbb{Q}$-divisor $D_{\overline{\Bbbk}}$ on $S_{\overline{\Bbbk}}$ such that
$$
D_{\overline{\Bbbk}}\sim_{\mathbb{Q}} -K_{S_{\overline{\Bbbk}}}.
$$
Moreover, the pair $(S_{\overline{\Bbbk}},D_{\overline{\Bbbk}})$ is not log canonical at some point $P\in S_{\overline{\Bbbk}}$.
Furthermore, it follows from \cite[Lemma~3.7]{Ch07a} that
this  pair is log canonical outside of the point $P$.
Thus, the point $P$ must be invariant under the action of the Galois group $\mathrm{Gal}(\overline{\Bbbk}/\Bbbk)$ on $S_{\overline{\Bbbk}}$,
hence corresponds to a $\Bbbk$-rational point of $S$, a contradiction.
\end{proof}

\begin{example}
\label{example:cubics}
Let $X$ be the cubic hypersurface in $\mathbb{P}^4_{\mathbb{Q}}$ that is given by
$$
5x^3+9y^3+10z^3+12w^3+uf_2(x,y,z,w,u)=0,
$$
where $f_2$ is a homogeneous polynomial of degree $2$.
Let $S_X$ be its hyperplane section that is given by $u=0$.
Then $X$ has at most canonical singularities, the surface $S_X$ is smooth,
and $X\setminus S_X$ is an affine cubic hypersurface in $\mathbb{A}^4_{\mathbb{Q}}$.
Moreover, Cassels and Guy proved in \cite{CG66} that the surface $S_X$ violates the Hasse principle.
In particular, it does not contain any $\mathbb{Q}$-rational points.
Then $X\setminus S_X$ is super-rigid by Corollary~\ref{corollary:cubics}.
\end{example}

\begin{theorem}[Theorem~B]
\label{theorem:criterion}
Let $X\setminus S_X$ be an affine Fano variety with completion $X$ and boundary~$S_X$.
If $\alpha(S_X,\mathrm{Diff}_{S_X}(0))\geqslant 1$, then the affine Fano variety $X\setminus S_X$ is super-rigid.
\end{theorem}
\begin{proof}
Suppose that $X\setminus S_X$ is not super-rigid.
Then Theorem~A below implies that the variety $X$ must carry a mobile linear system $\mathcal{M}_X$ such that
the singularities of the pair $(X, S_X+\lambda \mathcal{M}_X)$ are not log canonical along $S_X$ for  a positive rational number $\lambda$ such that
$$
K_X+S_X+\lambda \mathcal{M}_X\sim_{\mathbb{Q}} 0.
$$
However the condition $\alpha(S_X,\mathrm{Diff}_{S_X}(0))\geqslant 1$ and Inversion of adjunction (see \cite[Theorem~7.5]{Ko97}) immediately show that this is impossible.
\end{proof}

\medskip

Using Theorem~B we are able to present many  examples of super-rigid affine Fano varieties $X\setminus S_X$
with a completion $X$ and a boundary $S_X$ such that $X$ is a weighted projective space, and $S_X$ is a hypersurface in $X$ that is well-formed and quasismooth (see \cite{Fletcher}).
To be precise, let $X$ be a well-formed weighted projective spaces $\mathbb{P}(a_0,a_1,\ldots,a_n)$,
where  $a_0,a_1,\ldots,a_n$ are positive integers such that $a_0\leqslant a_1\leqslant\cdots\leqslant a_n$.
Let $S_X$ be a quasi-smooth well-formed hypersurface in $X$ of degree
$$
d<\sum_{i=0}^{n}a_i.
$$
Then $X\setminus S_X$ is an affine Fano variety of dimension $n$, and $\mathrm{Diff}_{X}(0)=0$.

The example below is motivated by Conjecture~\ref{conjecture:cubics}.

\begin{example}
\label{example:hypersurface}
Suppose that $X=\mathbb{P}^{n}$ and $d=n\geqslant 4$. Then
$$
\mathrm{Aut}\left(X\setminus S_X\right)=\mathrm{Aut}\left(X,S_X\right),
$$
because  $S_X$ is not ruled (see \cite{IM71,Pukhlikov1998,Fernex}).
We do not know whether the affine Fano variety $X\setminus S_X$ is super-rigid or not.
However, if $n\geqslant 6$ and $S_X$ is a general hypersurface of degree~$n$, then $\alpha(S_X)=1$ by \cite{Pukhlikov2005},
so that $X\setminus S_X$ is super-rigid by Theorem~B.
\end{example}

Smooth cubic surfaces in $\mathbb{P}^3$, smooth quartic surfaces in $\mathbb{P}(1,1,1,2)$,
and smooth sextic surfaces in $\mathbb{P}(1,1,2,3)$ form three special families
of a much larger class of quasismooth well-formed two-dimensional del Pezzo hypersurfaces in three-dimensional weighted projective spaces.
They provide a lot of examples of super-rigid affine Fano threefolds.

\begin{example}
\label{example:exceptional}
Suppose that $n=3$, so that $X\setminus S_X$ is an affine Fano threefold.
Then it follows from \cite{CPS10,CS13} that $\alpha(S_X)>1$ if and only if
$(a_{0},a_{1},a_{2},a_{3},d)$ is one of the following the quintuples:
\begin{longtable}{llll}
$(2,3,5,9,18)$, &$(3,3,5,5,15)$, &$(3,5,7,11,25)$, &$(3,5,7,14, 28)$,\\%
$(3,5,11,18, 36)$, &$(5,14,17,21,56)$, &$(5,19,27,31,81)$, &$(5,19,27,50,100)$,\\
$(7,11,27,37,81)$, &$(7,11,27,44,88)$, &$(9,15,17,20,60)$, &$(9,15,23,23,69)$,\\
$(11,29,39,49,127)$, &$(11,49,69,128,256)$, &$(13,23,35,57,127)$, & $(13,35,81,128,256)$, \\
$(3,4,5,10,20)$, &$(3,4,10,15,30)$, &$(5,13,19,22,57)$, &$(5,13,19,35,70)$, \\
$(6,9,10,13,36)$, &$(7,8,19,25,57)$, &$(7,8,19,32,64)$, &$(9,12,13,16,48)$, \\
$(9,12,19,19,57)$, &$(9,19,24,31,81)$, &$(10,19,35,43,105)$,& $(11,21,28,47,105)$, \\
$(11,25,32,41,107)$, &$(11,25,34,43,111)$, &$(11,43,61,113,226)$, &$(13,18,45,61,135)$, \\
$(13,20,29,47,107)$, &$(13,20,31,49,111)$, &$(13,31,71,113,226)$, &$(14,17,29,41,99)$,\\
$(5,7,11,13,33)$, &$(5,7,11,20,40)$, &$(11,21,29,37,95)$, &$(11,37,53,98,196)$, \\
$(13,17,27,41,95)$, &$(13,27,61,98,196)$, &$(15,19,43,74,148)$, &$(9,11,12,17,45)$, \\
$(10,13,25,31,75)$, &$(11,17,20,27,71)$, &$(11,17,24,31,79)$, &$(11,31,45,83,166)$, \\
$(13,14,19,29,71)$, &$(13,14,23,33,79)$, &$(13,23,51,83,166)$, &$(11,13,19,25,63)$,\\
$(11,25,37,68,136)$, &$(13,19,41,68,136)$, &$(11,19,29,53,106)$, &$(13,15,31,53,106)$,\\
$(11,13,21,38,76).$&&&\\
\end{longtable}
In all these cases, the affine Fano variety $X\setminus S_X$ is super-rigid by Theorem~B.
\end{example}

\begin{example}
\label{example:weakly-exceptional}
With the notation and assumptions of Example~\ref{example:exceptional},
one has $\alpha(S_X)=1$ if $(a_{0},a_{1},a_{2},a_{3},d)$ is one of the following the quintuples:
\begin{longtable}{llll}
$(1,3,5,8,16)$, &$(2,3,4,7,14)$, &
$(5,6,8,9,24)$, &$(5,6,8,15,30)$,\\
\multicolumn{2}{l}{$(2,2n+1,2n+1,4n+1,8n+4),$}&  \multicolumn{2}{l}{$(3,3n,3n+1,3n+1, 9n+3)$,} \\
\multicolumn{2}{l}{$(3, 3n+1, 3n+2, 3n+2, 9n+6)$,}& \multicolumn{2}{l}{$(3,3n+1,3n+2,6n+1, 12n+5)$,} \\
\multicolumn{2}{l}{$(3,3n+1, 6n+1, 9n, 18n+3)$,} & \multicolumn{2}{l}{$(3, 3n+1, 6n+1, 9n+3, 18n+6)$,} \\
\multicolumn{2}{l}{$(4, 2n+1, 4n+2, 6n+1, 12n+6)$,} & \multicolumn{2}{l}{$(4, 2n+3, 2n+3, 4n+4, 8n+12)$,} \\
\multicolumn{2}{l}{$(6, 6n+3, 6n+5, 6n+5, 18n+15)$,} & \multicolumn{2}{l}{$(6, 6n+5, 12n+8, 18n+9, 36n+24)$,} \\
\multicolumn{2}{l}{$(6, 6n+5, 12n+8, 18n+15, 36n+30)$,} & \multicolumn{2}{l}{$(8, 4n+5, 4n+7, 4n+9, 12n+23)$,} \\
\multicolumn{2}{l}{$(9, 3n+8, 3n+11, 6n+13, 12n+35)$,} & \\
\end{longtable}
\noindent where $n$ is any positive integer. Moreover, we also have $\alpha(S_X)=1$ if $S_X$ is a general hypersurface of degree $d$ in $\mathbb{P}(a_0,a_1,a_2,a_3)$  and $(a_{0},a_{1},a_{2},a_{3},d)$ is one of the following quintuples: $(1,1,2,3,6)$, $(1,2,3,5,10)$, $(1,3,5,7,15)$, $(2,3,4,5,12)$.
In all these cases, the affine Fano variety $X\setminus S_X$ is super-rigid by Theorem~B.
\end{example}

Smooth quartic threefolds in $\mathbb{P}^4$ form one family among the famous $95$ families of Reid and Iano-Fletcher discovered in \cite{Fletcher}.
They also provide many examples of super-rigid affine Fano fourfolds through Theorem~B.

\begin{example}
\label{example:95}
For the case where $n=4$, $d=\sum_{i=0}^{4}a_i-1$,  Iano-Fletcher verified that there are exactly $95$  quintuples $(a_0, a_1, a_2, a_3, a_4)$  that define the hypersurface $S_X$ with only terminal singularities and listed such quintuples in  \cite{Fletcher}.
Moreover, it follows from \cite{CPR,CP17} that
$$
\mathrm{Aut}\left(X\setminus S_X\right)=\mathrm{Aut}\left(X, S_X\right).
$$
Furthermore, if $-K_{S_X}^3\leqslant 1$ and the hypersurface $S_X$ is general,
then $\alpha(S_X)=1$ by \cite{Ch08,Ch09},
so that the corresponding affine Fano fourfold $X\setminus S_X$ is super-rigid by Theorem~B.
\end{example}

Meanwhile, Johnson and Koll\'ar completely described  the quintuples $(a_0,a_1,a_2,a_3,a_4)$ that define quasi-smooth hypersurface $S_X$ of degree  $\sum_{i=0}^{4}a_i-1$ in $\mathbb{P}(a_0,a_1,a_2,a_3,a_4)$ in \cite{JohnsonKollar}.
They also show that $\alpha(S_X)\geqslant 1$ in many these cases. In such cases, the corresponding affine Fano variety $X\setminus S_X$
is super-rigid by Theorem~B.


\section{Global finite quotients of affine spaces}
\label{section:quotients}
As seen in the previous section, it is hardly expected that the $\alpha$-invariant plays a role of a criterion (a sufficient and necessary condition) for an affine Fano variety to be super-rigid. However, it is able to serve as a criterion for affine Fano varieties of  a certain type. In this section, we study affine Fano varieties that are quotients of $\mathbb{A}^n$ by actions of finite subgroups of $\mathrm{GL}_n(\mathbb{C})$. We verify that the $\alpha$-invariant completely determines super-rigidity of such affine Fano varieties of dimensions up to $4$.

Let $G$ be a finite subgroup in $\mathrm{GL}_n(\mathbb{C})$, where $n\geqslant 2$. Put $\mathbb{V}=\mathbb{A}^n$ and consider $\mathbb{V}$ as a linear representation of the group $G$. Moreover, let us identify $\mathrm{GL}_n(\mathbb{C})$ with a subgroup of $\mathrm{PGL}_{n+1}(\mathbb{C})=\mathrm{Aut}(\mathbb{P}^n)$ by means of the natural embedding $\mathrm{GL}_n(\mathbb{C})\hookrightarrow\mathrm{GL}_{n+1}(\mathbb{C})$
and the natural projection $\mathrm{GL}_{n+1}(\mathbb{C})\to\mathrm{PGL}_{n+1}(\mathbb{C})$.
In this way, we thus consider $G$ as a subgroup of $\mathrm{Aut}(\mathbb{P}^n)$.

Let $H\cong\mathbb{P}^{n-1}$ be the $\mathrm{GL}_n(\mathbb{C})$-invariant hyperplane  $\mathbb{P}^n\setminus \mathbb{V}$.
The action of $G$ on $H$ is not necessarily faithful. Denote by $\overline{G}$ the image of $G$ in $\mathrm{Aut}(H)=\mathrm{PGL}_{n}(\mathbb{C})$, and denote by $Z$ the kernel of the group homomorphism $G\to\overline{G}$.
Then $Z$ is a cyclic group of order $m\geqslant 1$, and $G$ is a central extension of the group $\overline{G}$ by $Z$.
Letting
$$
X=\mathbb{P}^n/G,
$$
we can identify the quotient $S_X=H/\overline{G}$ with a prime divisor in $X$,
so that by construction
\begin{equation}
\label{equation:quotient}
X\setminus S_X\cong\mathbb{A}^n/G.
\end{equation}

Recall that an element $g\in G\subset\mathrm{GL}_{n}(\mathbb{C})$ is called a \emph{quasi-reflection} if there is a hyperplane in $H\cong\mathbb{P}^{n-1}$ that is pointwise fixed by the image of $g$ in $\overline{G}$. If $G$ is generated by quasi-reflections, then by virtue of the Chevalley-Shephard-Todd theorem $\mathbb{A}^n/G$ is isomorphic to $\mathbb{A}^n$, hence is in particular an affine Fano variety. More generally, we have:

\begin{lemma}
\label{lemma:globquot-affFano} The quotient \eqref{equation:quotient} is an affine Fano variety with completion $X=\mathbb{P}^n/G$ and boundary $S_X=H/\overline{G}$.
\end{lemma}

\begin{proof}
Since the subgroup of $G$ generated by quasi-reflections is normal, we may assume that $G$ does not contain non trivial quasi-reflection. Note that $G$ when considered as a subgroup in $\mathrm{GL}_{n+1}(\mathbb{C})$ may contain quasi-reflections. These are just elements of $Z$ that are different from identity. To take these quasi-reflexions into account, we consider the commutative diagram
$$
\xymatrix{
\mathbb{P}^n\ar@{->}[rd]_{\pi}\ar@{->}[rr]^{q}&&\mathbb{P}(1^{n},m)\ar@{->}[dl]^{\overline{\pi}}\\%
&X&}
$$ %
where $\pi$ is the quotient map by the group $G$, the morphism $q$ is the quotient map by the group $Z$, and $\overline{\pi}$ is the quotient map by the group $\overline{G}$.
By construction, the finite morphism $q$ is branched over $q(H)$, which is a smooth hypersurface in $\mathbb{P}(1^{n},m)$ of degree $n$, that does not contain the singular point of the weighted projective space $\mathbb{P}(1^{n},m)$.
Moreover, the finite morphism $\overline{\pi}$ is unramified in codimension one.
This implies that the divisor
$$
-(K_{X}+S_X)
$$
is ample, and $(X,S_X)$ has purely log terminal singularities.
Thus, the quotient \eqref{equation:quotient} is an affine Fano variety as desired.
\end{proof}

In view of the above discussion, to address the question whether the quotient \eqref{equation:quotient} is a super-rigid affine Fano variety, we may and will assume in the sequel that $G$ is \emph{small}, i.e., does not contain any non trivial quasi-reflection.
Let $\alpha_{\overline{G}}(H)$ be the number defined as
$$
\mathrm{sup}\left\{\ \lambda\in\mathbb{Q}\ \left| \
 \aligned
& \mbox{the  pair } (H,\lambda D) \mbox{ is log canonical for every  effective}\\
& \mbox{ $\overline{G}$-invariant $\mathbb{Q}$-divisor $D$ on $H$ such that } D\sim_{\mathbb{Q}} -K_{H}.\\
\endaligned\right.\right\}.
$$
Then $\alpha_{\overline{G}}(H)$ is the $\overline{G}$-equivariant $\alpha$-invariant of Tian of the projective space $H\cong\mathbb{P}^{n-1}$.
Moreover, it follows from the proof of \cite[Theorem~3.16]{CS11} that
\begin{equation}
\label{eq:alpha}
\alpha\left(S_X,\mathrm{Diff}_{S_X}(0)\right)=\alpha_{\overline{G}}(H).
\end{equation}
Furthermore, one has $\alpha_{\overline{G}}(H)\geqslant 1$ if and only if
the quotient singularity $\mathbb{A}^n/G$ is weakly-exceptional in the notation of \cite{CS11,CS14,Sakovics2012,Sakovics2014}.
In particular, if $\alpha_{\overline{G}}(H)\geqslant 1$,
then $\mathbb{V}$ must be an irreducible representation of the group $G$ by \cite[Theorem~3.18]{CS11}.
In this case, the subgroup $\overline{G}\subset\mathrm{Aut}(H)$ is usually called \emph{transitive}.

\begin{lemma}
\label{lemma:criterion2}
One has $\alpha_{\overline{G}}(H)\geqslant 1$ $\Rightarrow$  \eqref{equation:quotient} is super-rigid $\Rightarrow$ $\overline{G}$ is transitive.
\end{lemma}

\begin{proof}
The first implication follows from Theorem B and \eqref{eq:alpha}.
For the second one, suppose on the contrary that $\mathbb{V}$ splits as the direct sum of two nontrivial representations
$$
\mathbb{V}=\mathbb{V}_1\oplus \mathbb{V}_2.
$$
Then the projection $\mathrm{pr}_1\colon\mathbb{V}\rightarrow \mathbb{V}_1$ descends to a fibration
$\rho\colon X\setminus S_X\rightarrow\mathbb{V}_1/G$,
whose general fibers are isomorphic to  $\mathbb{V}_2/G^\prime$,
where $G^\prime\subset G$ denotes the stabilizer of a general fiber of $\mathrm{pr}_1$.
Since $\mathbb{V}_2/G^\prime$ is an affine Fano variety by Lemma \ref{lemma:globquot-affFano},
we see that the generic fiber of $\rho$ is an affine Fano variety. In other words, $\rho:X\setminus S_X\rightarrow \mathbb{V}_1/G$ is a relative affine Fano variety in contradiction to the super-rigidity hypothesis.
\end{proof}

\begin{corollary}
\label{corollary:Nadel}
Suppose that for every irreducible $\overline{G}$-invariant subvariety $Z\subset H$, there exists no hypersurface in $H\cong\mathbb{P}^{n-1}$ of degree $\mathrm{dim}(Z)+1$ that contains $Z$. Then $X\setminus S_X$ is super-rigid.
\end{corollary}

\begin{proof}
Indeed, the conditions imply that $\alpha_{\overline{G}}(H)\geqslant 1$ by virtue of \cite[Theorem~1.12]{CS14}.
\end{proof}

One can show that super-rigid affine Fano quotients \eqref{equation:quotient} exist in all dimensions.

\begin{example}
\label{example:Heisenberg}
Suppose that $n$ is an odd prime. Let $G$ is a subgroup in $\mathrm{SL}_n(\mathbb{C})$ that is isomorphic to the Heisenberg group of order $n^3$.
Then $\alpha_{\overline{G}}(H)\geqslant 1$ by \cite[Theorem~1.15]{CS14}, so that $X\setminus S_X$ is super-rigid by Lemma~\ref{lemma:criterion2}.
\end{example}

We now give a complete classification of super-rigid affine Fano varieties \eqref{equation:quotient} of dimensions at most $4$.

\begin{theorem}
\label{theorem:quotient-n-2}
Suppose that $n=2$. Then the following conditions are equivalent:
\begin{enumerate}
\item The affine Fano variety \eqref{equation:quotient} is super-rigid.
\item The inequality $\alpha_{\overline{G}}(H)\geqslant 1$ holds.
\item The curve $H\cong\mathbb{P}^1$ does not have $G$-fixed points.
\item The representation $\mathbb{V}$ is irreducible.
\item The group $G$ is not abelian.
\item The group $G$ is not cyclic.
\end{enumerate}
\end{theorem}

\begin{proof}
We have $S_X\cong H\cong\mathbb{P}^1$.
For every point $O\in H$,
denote by $n_P$ the order of the stabilizer in $\overline{G}$ of any point $O\in H$ that is mapped to $P$ by the quotient map $\pi\colon\mathbb{P}^2\rightarrow X$. Then
$$
\mathrm{Diff}_{S_X}(0)=\sum_{P\in S_X}\frac{n_{P}-1}{n_{P}}P.
$$
Note that $\mathrm{Diff}_{X}(0)\ne 0$ provided that $\overline{G}$ is not trivial.
Let $Q$ be the point in $S_X$ with the largest $n_Q$. Then it follows from \cite[Example~3.3]{CS11} that
$$
\alpha\left(S_X,\mathrm{Diff}_{S_X}(0)\right)=\frac{1-\frac{n_{Q}-1}{n_{Q}}}{2-\sum_{P\in S_X}\frac{n_{P}-1}{n_{P}}}=
\left\{%
\aligned
&6\ \text{if}\ \overline{G}\cong\mathfrak{A}_{5},\\%
&3\ \text{if}\ \overline{G}\cong\mathfrak{S}_{4},\\%
&2\ \text{if}\ \overline{G}\cong\mathfrak{A}_{4},\\%
&1\ \text{if}\ \overline{G}\ \text{is a dihedral group},\\%
&1\ \text{if}\  \overline{G}\cong\mathbb{Z}_{2}\times\mathbb{Z}_2,\\%
&\frac{1}{2}\ \text{if}\  \overline{G}\ \text{is cyclic}.\\%
\endaligned\right.
$$
Furthermore, since we assumed that $G$ is small, it is abelian if and only if it is cyclic. Finally, observe that $\alpha_{\overline{G}}(H)$ is the length of the smallest $\overline{G}$-orbit in $H$, so that
$$ \alpha\left(S_X,\mathrm{Diff}_{S_X}(0)\right)=\alpha_{\overline{G}}(H)=\frac{|\overline{G}|}{n_Q} $$
by virtue of \eqref{eq:alpha}.
Now the conclusion follows from Lemma~\ref{lemma:criterion2} together with the fact that the representation $\mathbb{V}$ splits when $G$ is cyclic.
\end{proof}

Using Theorem~\ref{theorem:quotient-n-2} together with classical results of Miyanishi and Sugie on $\mathbb{Q}$-homology planes with quotient singularities \cite{MiSu}, we obtain the following classification of all two-dimensional super-rigid affine Fano varieties.

\begin{theorem}
\label{theorem:del-Pezzo}
Let $S\setminus C$ be an affine Fano variety of dimension $2$ with completion $S$ and boundary $C$.
Then $S\setminus C$ is super-rigid if and only if the pair $(S,C)$ is isomorphic to a pair
$$
\left(\mathbb{P}^2/G,H/\overline{G}\right)
$$
for some non-cyclic small finite subgroup $G\subset\mathrm{GL}_2(\mathbb{C})$.
\end{theorem}

\begin{proof}
Since $(S,C)$ has purely log terminal singularities, $-(K_S+C)$ is ample and the Picard rank of $S$ is equal to $1$,
it follows that $S\setminus C$ is a logarithmic $\mathbb{Q}$-homology plane with smooth locus of negative Kodaira dimension $-\infty$.
By \cite[Theorems~2.7~and~2.8]{MiSu}, the surface $S\setminus C$ either contains an open $\mathbb{A}^1$-cylinder, which is impossible as it is super-rigid by hypothesis,
or is isomorphic to a quotient $\mathbb{A}^2/G$ for a small finite subgroup $G$ of $\mathrm{GL}_2(\mathbb{C})$.

Since $S\setminus C$ is super-rigid, the group $G$ is not cyclic by Theorem  \ref{theorem:quotient-n-2}.

The quotient space $\mathbb{P}^2/G$  is  the natural projective completion of $\mathbb{A}^2/G$ with boundary $H/\overline{G}$.
The isomorphism $S\setminus C\cong \mathbb{A}^2/G$ extends to a birational map $S\dashrightarrow \mathbb{P}^2/G$,
which must be an isomorphism of pairs $(S,C)\cong (\mathbb{P}^2/G,H/\overline{G})$ by the definition of super-rigidity.
\end{proof}

A consequence of Theorem \ref{theorem:quotient-n-2} is that for $n=2$ all three conditions of Lemma~\ref{lemma:criterion2} are actually equivalent. This is no longer true for $n\geqslant 3$ as illustrated by the following example.

\begin{example}
\label{example:quotient-quadric-n-3}
Suppose that $n=3$. Let $G=\mathfrak{A}_5$ and let $\mathbb{V}$ be an irreducible three-dimensional representation of $G$.
Then $\overline{G}\cong G$, and the center $Z$ is trivial.
Moreover, there exists a $\overline{G}$-invariant smooth conic $C\subset H$.
Let $\pi\colon W\to\mathbb{P}^3$ be the blow up of the conic $C$,
with exceptional divisor $E$, and let $\widetilde{H}$ be the proper transform of the plane $H$ on the threefold $W$.
Then there exists a $G$-equivariant commutative diagram
$$
\xymatrix{
W\ar@{->}[rd]_{\pi}\ar@{->}[rr]^{\eta}&& Q\ar@{-->}[dl]^{\psi}\\%
&\mathbb{P}^3&}
$$ %
where $Q$ is a smooth quadric threefold in $\mathbb{P}^4$,
the morphism $\eta$ is the contraction of the surface $\widetilde{H}$ to a smooth point of $Q$,
and $\psi$ is a linear projection from this point.
Then $(Q,\eta(E))$ is purely log terminal and $-(K_{Q}+\eta(E))$ is ample,
so that $(Q,\eta(E))$ is an affine Fano variety.
By construction, we have $Q\setminus\eta(E)\cong\mathbb{P}^3\setminus H$.
Let $Y=Q/G$ and $S_Y=\eta(E)/G$. Then $(Y,S_Y)$ is an affine Fano variety and
$Y\setminus S_Y\cong X\setminus S_X$, so that $(X,S_X)$ is not super-rigid.
\end{example}

In fact, for $n=3$ and $4$, we can obtain criteria for $X\setminus S_X$ to be super-rigid which are similar to Theorem~\ref{theorem:quotient-n-2}.
For example, we have:

\begin{theorem}
\label{theorem:quotient-n-3}
Suppose that $n=3$. Then the following conditions are equivalent:
\begin{enumerate}
\item The affine Fano variety $X\setminus S_X$ is super-rigid;
\item One has $\alpha_{\overline{G}}(H)\geqslant 1$;
\item The plane $H$ contains neither $\overline{G}$-invariant lines nor $\overline{G}$-invariant conics.
\end{enumerate}
\end{theorem}

\begin{proof}
By \cite[Theorem~3.23]{CS11}, the conditions (2) and (3) are equivalent.
By Theorem~B and \eqref{eq:alpha}, the condition (2) implies (1).
Thus, we have to show that (1) implies (3). If $H$ contains a $\overline{G}$-invariant line,
then $X\setminus S_X$ is not super-rigid by Lemma~\ref{lemma:criterion2}.
Similarly, if $H$ contains a $\overline{G}$-invariant conic, then the construction presented in Example~\ref{example:quotient-quadric-n-3}
shows that $X\setminus S_X$ is not super-rigid.
\end{proof}

Similarly, for $n=4$ we have

\begin{theorem}
\label{theorem:quotient-n-4}
Suppose that $n=4$. Then the following conditions are equivalent:
\begin{enumerate}
\item The affine Fano variety $X\setminus S_X$ is super-rigid;
\item One has $\alpha_{\overline{G}}(H)\geqslant 1$;
\item The group $\overline{G}$ is transitive,
and the hyperplane $H$ does not contain  $\overline{G}$-invariant quadrics,
$\overline{G}$-invariant cubic surfaces or $\overline{G}$-invariant twisted cubic curves.
\end{enumerate}
\end{theorem}

\begin{proof}
By \cite[Theorem 4.3]{CS11}, the conditions (2) and (3) are equivalent.
By Theorem~B and \eqref{eq:alpha}, the condition (2) implies (1).
Thus, we have to show that (1) implies (3). If the group $\overline{G}$ is not transitive,
then $X\setminus S_X$ is not super-rigid by Lemma~\ref{lemma:criterion2}.
Thus, we may assume that $\overline{G}$ is transitive.
Let us show that $X\setminus S_X$ is not super-rigid in the case where
the hyperplane $H$ contains one of the following $\overline{G}$-invariant subvarieties:
a quadric surface, a cubic surface, or a twisted cubic curve.

Suppose first that the hyperplane $H$ contains a $\overline{G}$-invariant surface $S_d$ of degree $d$ such that either $d=2$ or $d=3$.
Because $\overline{G}$ is transitive, this surface must be smooth. This is obvious in the case when $d=2$ while in the case when $d=3$, the conclusion follows from the classification of singular cubic surfaces.

We may assume that $H$ is given by $w=0$, and $S_d$ is given by
$$
\left\{\aligned%
&f_d(x,y,z,t)=0,\\
&w=0,\\
\endaligned
\right.
$$
where $f_d(x,y,z,t)$ is a homogeneous polynomial of degree $d$,
and $x$, $y$, $z$, $t$, $w$ are homogeneous coordinates on $\mathbb{P}^4$.
Let $V$ be a hypersurface in $\mathbb{P}(1^5,d-1)$ given by
$$
wu=f_d(x,y,z,t),
$$
where $x$, $y$, $z$, $t$, $w$, $u$ are quasi-homogeneous coordinates on $\mathbb{P}(1^5,d-1)$
such that $u$ be a coordinate of weight $d-1$.
Then there exists a $G$-equivariant commutative diagram
$$
\xymatrix{
W\ar@{->}[rd]_{\pi}\ar@{->}[rr]^{\eta}&& V\ar@{-->}[dl]^{\psi}\\%
&\mathbb{P}^4&}
$$ %
where $\pi$ is the blow up of the surface $S_d$, with exceptional divisor $E$,
the morphism $\eta$ is the contraction of the proper transform of the hyperplane $H$ to the point $[0:0:0:0:0:1]$,
and $\psi$ is a linear projection from this point.
Then $\eta(E)$ is cut out on $V$ by $w=0$. Let $Y=V/G$ and $S_Y=\eta(E)/G$. Then $(Y,S_Y)$ is an affine Fano variety and
$Y\setminus S_Y\cong X\setminus S_X$. In particular, we see that $X\setminus S_X$ is not super-rigid.

Next suppose that $H$ contains a $\overline{G}$-invariant twisted cubic curve $C$.
Then the group $\overline{G}$ acts faithfully on the curve $C$.
Using the classification of finite subgroups in $\mathrm{PGL}_{2}(\mathbb{C})$,
we see that $\overline{G}$ is isomorphic to one of the three groups $\mathfrak{A}_4$, $\mathfrak{S}_4$, $\mathfrak{A}_5$,
because we assumed that $\overline{G}$ is transitive.
Let $\sigma\colon U\to\mathbb{P}^4$ be the blow up of the curve $C$, with exceptional divisor $F$, and denote by $\widetilde{H}$
the proper transform of the hyperplane $H$ on the fourfold $U$.
Then there exists a $G$-equivariant morphism $\upsilon\colon V\to\mathbb{P}^2$ that is a $\mathbb{P}^1$-bundle (see \cite[Application~1]{SzWi90}).
Its fibers are the proper transforms of the secants of the curve $C$.
Moreover, it follows from \cite[Proposition~3.9]{ProkhorovZaidenberg}
that there exists a $G$-equivariant commutative diagram
$$
\xymatrix{
U\ar@{->}[rd]_{\sigma}\ar@{->}[rr]^{\nu}&& W_5\ar@{->}[dl]^{\phi}\\%
&\mathbb{P}^4&}
$$ %
where $W_5$ is a smooth del Pezzo fourfold in $\mathbb{P}^7$ of degree $5$,
the morphism $\nu$ is a contraction of the proper transform of $\widetilde{H}$ to a plane in $W_5$
such that the restriction $\nu\vert_{\widetilde{H}}\colon\widetilde{H}\to\nu(\widetilde{H})$ is the $\mathbb{P}^1$-bundle $\upsilon$,
and $\phi$ is the linear projection from the plane $\nu(\widetilde{H})$.
Then $\nu(F)$ is a singular hyperplane section of the fourfold $W_5$ such that the pair $(W_5,\nu(F))$ has purely log terminal singularities.
As above, we let $M=W_5/G$ and $S_M=\nu(F)/G$. Then $(M,S_M)$ is an affine Fano variety and
$ M\setminus S_M\cong X\setminus S_X$, so that $X\setminus S_X$ is not super-rigid.
\end{proof}


\section{Complements to del Pezzo surfaces}
\label{section:complements}

Let $S$ be a del Pezzo surface of anticanonical degree at most $3$, i.e., $K_S^2\leqslant 3$, with at worst Du Val singularities.
If the anticanonical degree is $3$, then $S$ is a cubic surface in $\mathbb{P}^3$.
Similarly, if the anticanonical degree is $2$, then $S$ is a quartic surface in $\mathbb{P}(1,1,1,2)$.
Finally, if the anticanonical degree is $1$, then $S$ is a sextic surface in $\mathbb{P}(1,1,2,3)$.

In this section, we study automorphism groups of affine Fano varieties whose boundary is $S$ and whose completion
is the corresponding ambient weighted projective space. Using the result obtained in this section we are able to yield many nontrivial non-examples of super-rigid affine Fano threefolds.

Let $\mathbb{P}$ be the ambient spaces $\mathbb{P}^{3}$, $\mathbb{P}(1,1,1,2)$ and $\mathbb{P}(1,1,2,3)$ in the cases of anticanonical degrees $3$, $2$ and $1$, respectively.
We have
$$
\mathbb{P}=\mathrm{Proj}\left(\mathbb{C}[x,y,z,w]\right),
$$
where the variables $x$, $y$, $z$ and $w$ are of weights $(1,1,1,1)$, $(1,1,1,2)$ and $(1,1,2,3)$ according to anticanonical degrees $3$, $2$ and $1$.
Denote by $d$ the degree of the surface $S$ as a hypersurface in $\mathbb{P}$.
Then $d$ is equal to $3$, $4$ and $6$ according to anticanonical degrees $3$,~$2$ and $1$, respectively.

The following theorem summarizes the current knowledge towards the structure of the automorphism groups of the  complements of smooth del Pezzo hypersurfaces.

\begin{theorem}
\label{theorem:smoothdP}
Suppose that $S$ is smooth. If its anticanonical degree is $1$, then
$$
\mathrm{Aut}\left(\mathbb{P}\setminus S\right)=\mathrm{Aut}\left(\mathbb{P},S\right)
$$
is a finite group. If its anticanonical degree is either $2$ or $3$, then $\mathrm{Aut}(\mathbb{P}\setminus S)$ does not contain nontrivial connected algebraic groups.
\end{theorem}

\begin{proof}
We follow the same strategy as \cite[$\S$ 2.2]{DuKi15},
which consists in shifting the question to a suitable finite \'etale cover of $\mathbb{P}\setminus S$.
Namely let $\pi\colon V\rightarrow \mathbb{P}$ be a cyclic Galois cover of degree $d$ branched along $S$ and \'etale elsewhere.
Then $V$ is a smooth Fano threefold of index $2$, isomorphic to a hypersurfaces of degree $d$ in $\mathbb{P}^{4}$, $\mathbb{P}(1,1,1,1,2)$ and $\mathbb{P}(1,1,1,2,3)$ in the cases $K_S^2=3$, $2$ and $1$, respectively.
Furthermore, the ramification divisor of $\pi$ coincides with a hyperplane section $H$ of $V$.
We then have a split exact sequence of groups
$$
0\rightarrow \mathrm{Aut}\left(V\setminus H,\pi\right)\rightarrow\mathrm{Aut}\left(V\setminus H\right)\rightarrow\mathrm{Aut}\left(\mathbb{P}\setminus S\right)\rightarrow 0,
$$
where $\mathrm{Aut}(V\setminus H,\pi)\cong\mathbb{Z}/d\mathbb{Z}$ denotes the group of deck transformations of the induced \'etale Galois cover
$\pi|_{V\setminus H}\colon V\setminus H\rightarrow \mathbb{P}\setminus S$.
The surjectivity of the right hand side homomorphism follows from the fact that
$$
\pi_*\mathcal{O}_{V\setminus H}\cong\bigoplus_{i=0}^{d-1}\left(\omega_{\mathbb{P}\setminus S}^{-1}\right)^{\otimes i}
$$
where $\omega_{\mathbb{P}\setminus S}^{-1}=\omega_{X}^{-1}|_{\mathbb{P}\setminus S} \cong \mathcal{O}_{\mathbb{P}}(d+1)|_{\mathbb{P}\setminus S}\cong\mathcal{O}_{\mathbb{P}}(1)|_{\mathbb{P}\setminus S}$ denotes the anticanonical sheaf of the variety $\mathbb{P}\setminus S$.

If $K_S^2=1$, then $\mathrm{Bir}(V)=\mathrm{Aut}(V)$ by \cite{Grinenko1,Grinenko2},
so that $\mathrm{Aut}(V\setminus H)=\mathrm{Aut}(V,H)$.
This implies in turn that $\mathrm{Aut}(\mathbb{P}\setminus S)=\mathrm{Aut}(\mathbb{P},S)$, which is a finite group.

If $K_S^2=2$ or $K_S^2=3$, then $V$ is not rational by \cite[Corollaire~4.8]{Voi88} and \cite{CG72}, respectively.
Now suppose that $\mathrm{Aut}(\mathbb{P}\setminus S)$ contains a nontrivial connected algebraic subgroup $G$. First note that $G$ is necessarily an affine algebraic group. Indeed, otherwise by Chevalley's theorem, there would exist  a maximal proper normal linear algebraic subgroup $G'$ of $G$ such that $G/G'$ is  abelian variety, and so the affine variety $\mathbb{P}\setminus S$ would inherit a nontrivial action of a proper connected algebraic group, which is absurd. Since it is affine and connected, $G$ thus contains either $\mathbb{G}_m$ or $\mathbb{G}_a$, which implies in turn that $V\setminus H$ admits a nontrivial action of either $\mathbb{G}_m$ or $\mathbb{G}_a$. By virtue of Rosenlicht's Theorem \cite{Ro56}, it would follow  that $V$ is birational to $Z\times \mathbb{P}^1$ for some rational surface $Z$. But then, the threefold $V$ itself would be rational, which is a contradiction.
\end{proof}

A consequence of Theorem~\ref{theorem:smoothdP} is the following generalization of \cite[Proposition 10]{DuKi15}.

\begin{corollary}
\label{cor:nocylinder}
If $S$ is smooth, then $\mathbb{P}\setminus S$ does not contain open $\mathbb{A}^1$-cylinders.
\end{corollary}

\begin{proof}
Observe that the divisor class group
$$
\mathrm{Cl}\left(\mathbb{P}\setminus S\right)\cong\mathbb{Z}/d\mathbb{Z}
$$
is finite. Thus, it follows from \cite[Proposition 2]{DuKi15} that for every open $\mathbb{A}^1$-cylinder $B\times\mathbb{A}^1$ in $\mathbb{P}\setminus S$, there exists an action of $\mathbb{G}_{a}$ on $\mathbb{P}\setminus S$ whose general orbits coincide
with the general fibers of the projection $\mathrm{pr}_{B}\colon B\times \mathbb{A}^1\rightarrow B$.
But on the other hand, the group $\mathrm{Aut}(\mathbb{P}\setminus S)$ does not contain any nontrivial connected algebraic group by Theorem \ref{theorem:smoothdP}.
\end{proof}

It is known that a smooth del Pezzo surface of anticanonical degree at most $3$ never contains any  $(-K_S)$-polar cylinder. For the singular case, we obtain a complete description for $(-K_{S})$-polar cylinders from \cite[Theorem~1.5]{CPW16}.
\begin{theorem}
\label{theorem:KS-polar}
 The surface $S$ does not contain any  $(-K_{S})$-polar cylinder if and only if
one of the following conditions is satisfied:
\begin{enumerate}
\item Its anticanonical degree is $1$  and  it has only singular points of types $\mathrm{A}_1$, $\mathrm{A}_2$, $\mathrm{A}_3$, $\mathrm{D}_4$ if any;
\item Its anticanonical degree is $2$  and  it allows only singular points of type $\mathrm{A}_1$ if any;
\item Its anticanonical degree is $3$  and  it allows no singular point.
\end{enumerate}
\end{theorem}

Existence of open $\mathbb{A}^1$-cylinders in the complements of singular normal cubic surfaces with singularities
strictly worse than $\mathrm{A}_2$ was first established in \cite[Proposition~3.7]{KPZ11}. A first example of nodal cubic surface whose complement contains an open $\mathbb{A}^1$-cylinder was constructed later on by Lamy (unpublished). Combined with the above results on $(-K_S)$-polar cylinders of singular del Pezzo surfaces and Example~\ref{example:Grinenko}, this leads to  anticipate that the complement of a surface with a $(-K_S)$-polar cylinder contains an open $\mathbb{A}^1$-cylinder. Ideas evolving from the proof of \cite[Theorem~4.1]{CPW16} and the study of open ``vertical'' $\mathbb{A}^1$-cylinders in del Pezzo fibrations \cite{DuKi16} turn out to confirm this expectation, namely:

\begin{theorem}[Theorem~C]
\label{theorem:cylinder}
If the surface $S$ contains an open $(-K_S)$-cylinder, then the affine Fano variety  $\mathbb{P}\setminus S$ contains an open $\mathbb{A}^1$-cylinder.
\end{theorem}

Before we proceed the proof, let us first prepare setups for the proof.
Due to  Theorem~\ref{theorem:KS-polar} above the surface $S$ contains a singular point $P$. Furthermore, we may assume that the singular point $P$ is not of type $\mathrm{A}_1$ if the anticanonical degree of $S$ is $2$; it is not of types $\mathrm{A}_1$, $\mathrm{A}_2$, $\mathrm{A}_3$, $\mathrm{D}_4$ if  the anticanonical degree  is $1$. By suitable coordinate changes, we may assume that the point $P$ is located at $[1:0:0:0]$.
Under such conditions, we immediately observe the following:

\begin{lemma}
\label{lemma:polynomial-form}
Under the condition above, the surface $S$ is defined  in $\mathbb{P}$ by one quasi-homogenous equation of the following types:
\begin{itemize}
\item Case $K_S^2=3$.
\begin{equation}\label{eq:degree3}
xf_2(y,z,w)+f_{3}(y,z,w)=0,
\end{equation}
where $f_2(y,z,w)$ and $f_3(y,z,w)$ are homogenous polynomials of degrees $2$ and $3$.
\item Case $K_S^2=2$.
\begin{equation}\label{eq:degree2}
w^2+x\left(ayw+f_3(y,z)\right)+f_{4}(y,z)=0,
\end{equation}
where $f_3(y,z)$ and $f_4(y,z)$ are quasi-homogenous polynomials of degrees $3$ and $4$, respectively, and $a$ is a  constant.
\item Case $K_S^2=1$.
\begin{equation}\label{eq:degree1-y}
 w^2+x\left(ay^2w+f_5(y,z)\right)+f_{6}(y,z)=0
\end{equation}
or
\begin{equation}\label{eq:degree1-z}
w^2+x\left(zw+f_5(y,z)\right)+f_{6}(y,z)=0,
\end{equation}
where $f_5(y,z)$ and $f_6(y,z)$ are quasi-homogenous polynomials of degrees $5$ and $6$, respectively, and $a$ is a constant.
\end{itemize}
\end{lemma}

\begin{proof}
This is easy to check.
\end{proof}

Let $\pi:S\dasharrow \Pi$ be the projection from the point $P$ to the hyperplane $\Pi$ defined by~$x=0$ in $\mathbb{P}$, i.e., $\pi([x:y:z:w])=[y:z:w]$. The hyperplane $\Pi$ is isomorphic to $\mathbb{P}^2$, $\mathbb{P}(1,1,2)$, $\mathbb{P}(1,2,3)$ according to the anticanonical degrees, $3$, $2$, $1$, respectively. We denote by $g(y,z,w)$ the coefficient  quasi-homogenous polynomial of $x$ in each of the quasi-homogenous equations of \eqref{eq:degree3}, \eqref{eq:degree2}, \eqref{eq:degree1-y} and  \eqref{eq:degree1-z}, i.e.,   for the case $K_S^2=1$
 $$g(y,z,w)=ay^2w \ (\mbox{or }zw) +f_5(y,z),$$ for the case $K_S^2=2$ $$g(y,z,w)=ayw+f_3(y,z)$$ and  for the case $K_S^2=3$
$$g(y,z,w)=f_2(y,z,w).$$
Let $D$ be the divisor on $S$ cut by the equation $g(y,z,w)=0$.
In the case of $K_S^2=3$, the divisor $D$ consists of the lines in $S$ that pass through $P$. There are at most six such lines and they are defined by the system of homogenous equations
$$
\left\{\aligned%
&g(y,z,w)=0\\
&f_3(y,z,w)=0\\
\endaligned
\right.
$$
in $\mathbb{P}^3$.
In case of of $K_S^2=2$, $D$ consists of at most six curves passing through the point~$P$. They are defined by the system of quasi-homogeneous equations
$$
\left\{\aligned%
&g(y,z,w)=0\\
&w^2+f_4(y,z)=0\\
\endaligned
\right.
$$
in $\mathbb{P}(1,1,1,2)$.
Finally, in case of of $K_S^2=1$,
it consists of at most five curves passing through the point~$P$, which are defined by the system of quasi-homogeneous equations
$$
\left\{\aligned%
&g(y,z,w)=0\\
&w^2+f_6(y,z)=0\\
\endaligned
\right.
$$
in $\mathbb{P}(1,1,2,3)$. In each case, the number of curves in $D$ is completely determined by the number of points determined by the system of quasi-homogeneous equations on $\Pi$. Denote these curves by $L_1,\ldots, L_r$ in each case.
The map $\pi$ contracts each curve $L_i$ to a point on $\Pi$.
\medskip

Now we are ready to prove Theorem~C.

\bigskip
\emph{Proof of Theorem~C.}
\medskip

Lemma~\ref{lemma:polynomial-form} immediately implies that  the projection $\pi$ is a birational map.
Moreover, it induces an isomorphism
$$\tilde{\pi}: S\setminus \left(L_1\cup\cdots\cup L_r\right)\to \mathrm{Im}(\tilde{\pi}) \subset\Pi.$$
Let $\mathcal{C}$ be the curve on $\Pi$ defined by
$$g(y,z,w)=0.$$
Note that this can be reducible or non-reduced.

\medskip

\emph{Claim 1.} If $K_S^2=3$, then $\mathrm{Im}(\tilde{\pi})=\Pi\setminus \mathcal{C}$ and  there is an hyperplane section $H$ of $S$ such that $S\setminus\left( H\cup L_1\cup\cdots\cup L_r\right)$ is an $\mathbb{A}^1$-cylinder.
\medskip

Let $\varphi\colon\overline{S}\to S$ be the blow up at the point $P$.
Then there exists a commutative diagram
$$
\xymatrix{
&\overline{S}\ar@{->}[ld]_{\varphi}\ar@{->}[rd]^{\phi}&\\%
S\ar@{-->}[rr]^{\pi}&&\Pi}
$$ %
where $\pi$ is the linear projection from the point $P$,
and $\phi$ is the birational morphism that contracts exactly  the proper transforms of the lines $L_1,\ldots, L_r$.
Let $E$ be the exceptional divisor of the blow up $\varphi$.
Then the image of the curve $E$ by $\phi$ in $\Pi$ is the conic curve $\mathcal{C}$. Then $\mathcal{C}$ contains all the points $\pi(L_i)$.
If $P$ is an ordinary double point of the surface $S$, then $\mathcal{C}$ is a smooth conic.
Moreover, if $P$ is a singular point of the surface $S$ of type $\mathrm{A}_n$ for $n\geqslant 2$, then $\mathcal{C}$ consists of two distinct lines.
Finally, if $P$ is either of type $\mathrm{D}_n$ or of type $\mathrm{E}_6$, then $\mathcal{C}$ is a double line.
\medskip

If $\mathcal{C}$ is smooth, let $\ell$ be a general line in $\Pi$ that is tangent to $\mathcal{C}$.
If $\mathcal{C}$ is singular, let $\ell$ be a general line in $\Pi$ that passes through a singular point of the conic $\mathcal{C}$.
By a suitable coordinate change, we assume that $\ell$ is defined by the equation $y=0$ on $\Pi$. Let $H$ be the divisor on $S$ cut by the equation $y=0$ in $\mathbb{P}^3$.

Then
$$
S\setminus\left(H\cup L_1\cup\cdots\cup L_r\right)\cong\Pi\setminus\left(\mathcal{C}\cup\ell\right)\cong\left(\mathbb{A}^1\setminus\left\{0\right\}\right)\times\mathbb{A}^1,
$$
so that $S\setminus\left(H\cup L_1\cup\cdots\cup L_r\right)$ is an $\mathbb{A}^1$-cylinder.
\medskip

For the rest of cases we also denote by $\ell$ the curve defined by $y=0$ on $\Pi$.  In addition let $H$ be the divisor on $S$ cut  by $y=0$ in $\mathbb{P}$.
\medskip

\emph{Claim 2.} If $K_S^2=2$ or if $K_S^2=1$ and the surface $S$ is defined by a quasi-homogenous equation of
type~\eqref{eq:degree1-y}, then $S\setminus\left(H\cup L_1\cup\cdots\cup L_r\right)$ is  an $\mathbb{A}^1$-cylinder.
\medskip

As in the case $K_S^2= 3$, the isomorphism $\tilde{\pi}$ maps $S\setminus\left(H\cup L_1\cup\cdots\cup L_r\right)$ onto $\Pi\setminus\left(\mathcal{C}\cup\ell\right)$. Meanwhile the projection $\pi$ maps $S\setminus H$ onto $\Pi\setminus \ell\cong\mathbb{A}^2$. Therefore the affine variety
$S\setminus\left(H\cup L_1\cup\cdots\cup L_r\right)$ is isomorphic to the complement of the curve defined by $$aw+f_3(1,z)=0$$ for type~\eqref{eq:degree2}
 (resp. $aw+f_5(1,z)=0$ for type~\eqref{eq:degree1-y}) in $\Pi\setminus \ell\cong\mathbb{A}^2$. This immediately implies the claim.

\medskip

\emph{Claim 3.}  If $K_S^2=1$ and the surface $S$ is defined by a quasi-homogenous equation of
type~\eqref{eq:degree1-z}, let $H_z$ be the hyperplane section of $S$ cut by the equation $z=0$. Then $S\setminus \left(H_z\cup L_1\cup\cdots\cup L_r\right)$ is  an $\mathbb{A}^1$-cylinder. In particular, $S\setminus H_z$ contains an open $\mathbb{A}^1$-cylinder.

 \medskip

In this case, the isomorphism $\tilde{\pi}$ maps $S\setminus\left(H_z\cup L_1\cup\cdots\cup L_r\right)$ onto  $\Pi\setminus\left(\mathcal{C}\cup\ell_z\right)$, where~$\ell_z$  is the curve on $\Pi$ defined by $z=0$.  The projection $\pi$ maps $S\setminus H_z$ onto $\Pi\setminus\ell_z$. The affine variety $\Pi\setminus\ell_z$ is isomorphic to the quotient space $\mathbb{A}^2 / \boldsymbol\mu_2$, where the $\boldsymbol\mu_2$-action is given by $(y,w) \mapsto (-y,-w)$. Notice that we abuse our notations $y$  and $w$. The affine open subset $\Pi\setminus\left(\mathcal{C}\cup\ell_z\right)$ is an  $\mathbb{A}^1$-cylinder. To see this,  notice that $\Pi\setminus\left(\mathcal{C}\cup\ell_z\right)$  is
the complement of the image of the curve defined by $$w+f_5(y,1)=0$$ in $\Pi\setminus \ell_z\cong\mathbb{A}^2/ \boldsymbol\mu_2$.
Since $f_5(y,1)$ is an odd polynomial in $y$, the isomorphism $\psi :\mathbb{A}^2\to\mathbb{A}^2$ defined by $(y_1,w_1)=\psi(y, w)=(y, w+f_5(y,1))$ is $ \boldsymbol\mu_2$-equivariant for the action $(y_1,w_1)\mapsto (-y_1,-w_1)$ so that we have a $ \boldsymbol\mu_2$-equivariant diagram
$$
\xymatrix{
\mathbb{A}^2\ar@{->}[rd]\ar@{->}[rr]^{\psi}&& \mathbb{A}^2\ar@{->}[dl]\\%
&\mathbb{A}^2/\boldsymbol\mu_2&},
$$ %
where the morphisms from $\mathbb{A}^2$ to $\mathbb{A}^2/\boldsymbol\mu_2$  are the respective quotient maps. This shows that
 $\Pi\setminus\left(\mathcal{C}\cup\ell_z\right)$  is isomorphic to  the complement in $\mathbb{A}^2/ \boldsymbol\mu_2$ of the image of the curve defined by $w_1=0$ . Letting $s=w_1^2$, $t=y_1^2$ and $u=w_1y_1$, $\mathbb{A}^2/ \boldsymbol\mu_2$ is isomorphic the affine variety defined by the equation  $st=u^2$ in $\mathbb{A}^3$. The image in  $\mathbb{A}^2/ \boldsymbol\mu_2$ of curve defined by $w_1=0$ coincides with the curve defined by $s=u=0$. Its complement is thus isomorphic to $\mathbb{A}^1\setminus\{0\}\times \mathbb{A}^1$, which is an $\mathbb{A}^1$-cylinder as desired.
\medskip

From now on, let $\mathcal{H}$  (resp. $\mathcal{G}$) be the hyperplane defined by $y=0$ (resp. $z=0$)  in $\mathbb{P}$ and  let $\mathcal{Q}$ be the (weighted) hypersurface (possibly reducible or non-reduced) defined by $g(y,z,w)=0$ in $\mathbb{P}$.

We first consider the surface $S$ dealt in Claims~1 and ~2.

From Claims~1 and 2 we obtain the property that
$$
S\setminus\left(\mathcal{H}\cup\mathcal{Q}\right)\cong \Pi\setminus\left(\mathcal{C}\cup\ell\right)
$$
is an $\mathbb{A}^1$-cylinder. In addition we suppose that the weighted surface $S$ is defined by $f(x,y,z,w)=0$ in $\mathbb{P}$. The equation must be one of the types in Lemma~\ref{lemma:polynomial-form}.

The pencil on $\mathbb{P}$ generated by the surfaces $S$ and $\mathcal{H}+\mathcal{Q}$  consists of (weighted) surfaces with equations of the form
$$
\alpha f(x,y,z,w)+\beta yg(y,z,w)=0,
$$
where $[\alpha:\beta]\in\mathbb{P}^{1}$.
This pencil gives a rational map $\rho\colon\mathbb{P}\dashrightarrow\mathbb{P}^{1}$.
Its generic fiber $S_{\eta}$ is the weighted surface defined by equation
$$
f(x,y,z,w)+\lambda yg(y,z,w)=0
$$
in the corresponding weighted projective space $\mathbb{P}_{\Bbbk(\lambda)}$
over the field $\Bbbk(\lambda)$, where $\lambda$ is a parameter.
It is singular at the point $[1:0:0:0]$,
and the intersection of $S_{\eta}$ with the surface in $\mathbb{P}_{\Bbbk(\lambda)}$ given by $f(x,y,z,w)=0$ consists of  the $\Bbbk(\lambda)$-rational hyperplane section
$$
\mathcal{H}_{\eta}=S_{\eta}\cap\{y=0\}
$$
and the $\Bbbk(\lambda)$-rational (weighted) hypersurface
section
$$
\mathcal{Q}_{\eta}=S_{\eta}\cap\{g(y,z,w)=0\}.
$$
As in the case of $S\setminus (\mathcal{H}\cup\mathcal{Q})$,
the complement $S_{\eta}\setminus (\mathcal{H}_{\eta}\cup\mathcal{Q}_{\eta})$ is an $\mathbb{A}^1$-cylinder in $S_{\eta}$,
defined over the field $\Bbbk(\lambda)$, and hence $\mathbb{P}\setminus S$ contains an open $\mathbb{A}^1$-cylinder (see \cite[Lemma 3]{DuKi16}).

Now we consider the surface $S$ defined by a quasi-homogenous equation of type~\eqref{eq:degree1-z}. Letting $S_{\eta}$ be the generic member of the pencil of hypersurfaces of degree $6$ generated by $S$ and $3\mathcal{G}$, we deduce from Claim~3 in the same way as in the previous cases that the complement of $S_{\eta}\setminus \mathcal{G}$ contains an open $\mathbb{A}^1$-cylinder defined over the field $\Bbbk(\lambda)$, hence that $\mathbb{P}\setminus S$ contains an open $\mathbb{A}^1$-cylinder.
\qed

\begin{corollary}
\label{corollary:del-Pezzo}
If $S$ contains a $(-K_S)$-polar cylinder, then $\mathrm{Aut}(\mathbb{P}\setminus S)\ne \mathrm{Aut}(\mathbb{P},S)$.
\end{corollary}

\begin{proof}
Suppose that $S$ contains a $(-K_S)$-polar cylinder. Using Theorems~\ref{theorem:KS-polar} and~C,
we see that the affine Fano variety $X\setminus S$ contains an open $\mathbb{A}^1$-cylinder.
Then we conclude in the same way as in the proof of Corollary~\ref{cor:nocylinder}
that the existence of an open $\mathbb{A}^1$-cylinder in $\mathbb{P}\setminus S$ implies the existence of a nontrivial action of
$\mathbb{G}_a$ on $\mathbb{P}\setminus S$. This implies that $\mathrm{Aut}(\mathbb{P}\setminus S)\ne\mathrm{Aut}(\mathbb{P},S)$, because the right hand side is a finite group.
\end{proof}

\begin{remark} The proof of Theorem~C reproves \cite[Theorem~1.5]{CPW16} in a uniform manner. The original proof for anticanonical degrees $1$ and $2$ in \cite{CPW16} has been given in a case-by-case way.

\end{remark}

Combined with Corollary \ref{cor:nocylinder}, Theorem~C fully settles the question of existence of open $\mathbb{A}^1$-cylinders in the complements to cubic surfaces with at worst Du Val singularities. Since the cubic surface $S$ has a $(-K_{S})$-polar cylinder if and only if $S$ is singular,  it follows from Corollary~\ref{cor:nocylinder} and Theorem~C
that the following two conditions are equivalent:
\begin{itemize}
\item The affine Fano variety $\mathbb{P}\setminus S$ contains an open $\mathbb{A}^1$-cylinder;
\item The cubic surface $S$ contains a $(-K_{S})$-polar cylinder.
\end{itemize}
It is natural to expect that the same holds in the cases of lower anticanonical degrees.
\begin{conjecture}
Let $S$ be a del Pezzo surface of anticanonical degree at most $2$ with at worst Du~Val singularities.
The affine Fano variety $\mathbb{P}\setminus S$ contains an open $\mathbb{A}^1$-cylinder if and only if  $S$ contains a $(-K_{S})$-polar cylinder.
\end{conjecture}

\begin{remark}
\label{remark:Voisin}
Using the idea of the proof of Theorem~\ref{theorem:smoothdP},
one can try to prove that the affine Fano variety $\mathbb{P}\setminus S$ does not contain open $\mathbb{A}^1$-cylinders
provided that $K_S^2\leqslant 2$ and $S$ does not contain $(-K_S)$-polar cylinders.
The crucial difference in this case is that the threefold $V$ in the proof of Theorem~\ref{theorem:smoothdP}
would no longer be smooth, because $S$ may be singular.
Hence, we cannot use \cite{Grinenko1,Grinenko2} and \cite[Corollaire~4.8]{Voi88} anymore.
For instance, if $K_S^2=2$ and $S$ is a singular quartic surface in $\mathbb{P}(1,1,1,2)$ that has at most ordinary double points,
then the threefold $V$ is a \emph{singular} double cover of $\mathbb{P}^3$ ramified in a quartic surface,
which has singular points of type $\mathrm{A}_3$.
The rationality problem for such singular quartic double solids have never been addressed as far as we are aware.
A priori, one can adapt \cite{Voi88,CPS2015} to prove their irrationality in some cases.
\end{remark}


\medskip

\textbf{Acknowledgements.}
This work was initiated during Affine Algebraic Geometry Meeting, in Osaka, Japan in March 2016, where three authors met. It was carried out during the authors' visits to Universit\'e de Bourgogne, Dijon France in June 2016, November 2017, and the first author's stay at the Max Planck Institute for Mathematics in~2017. This work was finalised during the workshop, The Shokurovs: Workshop for Birationalists, in Pohang, Korea in December 2017.
The third author was supported by IBS-R003-D1, Institute for Basic Science in Korea.

\end{document}